\documentclass[12pt, reqno]{amsart}
\usepackage[utf8]{inputenc}
\usepackage{amsmath}
\usepackage{amsfonts}
\usepackage{amssymb}
\usepackage{amsthm}
\usepackage{bm}
\usepackage{bbm}
\usepackage{xcolor}
\usepackage{hyperref}
\numberwithin{equation}{section}
\theoremstyle{plain}
\newtheorem{theorem}{Theorem}[section]
\newtheorem{lemma}[theorem]{Lemma}
\newtheorem{corollary}[theorem]{Corollary}

\newtheorem{conjecture}[theorem]{Conjecture}
\theoremstyle{definition}

\theoremstyle{remark}

\newcommand{\Z}{\mathbb{Z}}
\newcommand{\ex}{\mathcal{E}}

\title[Sums of primes, squares, and cubes]{Computational Results on Sums of a Prime with Squares or Cubes}
\address{Brigham Young University, Department of Mathematics, Provo, UT 84602, USA}
\author{Kenny Applegate}
\email{kapplegate2020@gmail.com}
\author{Kyle Pratt}
\email{kyle.pratt@mathematics.byu.edu}

\subjclass[2020]{11D85, 11P32, 11Y11}
\keywords{prime, square, cube, elliptic curve}

\allowdisplaybreaks

\date{}

\begin{document}

\begin{abstract}
One conjectures that all integers greater than 14 are the sum of an odd prime and two positive squares. We prove that all integers $n$ with $14<n<10^{26}$ can be represented in this manner. Along the way, we investigate sums of a prime and one square and then ``bootstrap'' to get our results on sums of a prime and two squares. We also show that all integers between $308$ and $10^{16}$ can be represented as the sum of an odd prime and two positive cubes. Our methods require the determination of the integral points on various elliptic curves.
\end{abstract}

\maketitle

\section{Introduction}
In additive number theory, one studies how numbers can be represented by summands from subsets of the integers. There are several famous examples of additive number theory problems, the most notable of which is probably Waring's problem (see \cite{VW2000} for a survey). Waring's problem asks how many nonnegative $k$th powers are needed to represent every positive integer as a sum of $k$th powers. It is known, for example, that every positive integer can be written as the sum of at most four squares or nine nonnegative cubes. In the case of cubes, it is conjectured that every sufficiently large integer may be written as the sum of at most four nonnegative cubes. Furthermore, while 454 requires at least eight cubes, every integer larger than 454 can be written as the sum of at most seven cubes (see \cite{Siksek2016} for further discussion).

Another well-known additive number theory problem is the Goldbach conjecture, which states that every even integer greater than two can be written as the sum of two primes. This conjecture is still wide open, but has been computationally verified for integers $\leq 4\cdot 10^{18}$ (see \cite{Goldbach2014}). Closely related to this is the ternary Goldbach conjecture, which states that every odd integer greater than seven can be written as the sum of three odd primes. In 1923, Hardy and Littlewood showed that the Generalized Riemann Hypothesis (GRH) implies the ternary Goldbach conjecture for all sufficiently large odd primes \cite{HL1923}. Vinogradov removed the dependence on GRH in 1937 (see \cite{Vin1937, Vin2004}, also \cite[Chapter 26]{Dav2000}). Finally, in 2013, Helfgott released a full proof of the conjecture (see \cite{Helfgott}), although revision and refereeing of his publication is still in progress.

The problem that we examine is a sort of mix of the two problems above. In particular, we look at which numbers can be represented as the sum of a prime and two $k$th powers, specifically for $k=2$ and $3$ (see, e.g., \cite[Conjecture J]{HL1923}, also the discussion in \cite[Introduction]{BW2026}). In 1957, Hooley \cite{Hooley1957} showed that all sufficiently large integers can be written as the sum of a prime and two squares, assuming GRH. Linnik removed the dependence on GRH in 1960 (see \cite[Chapter VII]{Linnik1963}). 

More precisely, one has the following conjecture.

\begin{conjecture}\label{conj:prime squares}
    If $n>14$, then $n=p+x^2+y^2$, where $p$ is an odd prime and $x, y$ are positive integers.
\end{conjecture}

In order to bridge the gap between 14 and the ``sufficiently large'' of Hooley and Linnik, we seek to computationally verify Conjecture \ref{conj:prime squares} as far as we can. To do this, we ``bootstrap'' from computational results on numbers that can be written as the sum of an odd prime and a positive square. We do this by a greedy algorithm, where we subtract the largest possible square from $n$ (or almost-largest squares) and show the difference can be written as the sum of a prime and a square. Since not all positive integers can be written as the sum of a prime and a square, we need to do a little more work. We draw inspiration from the following conjecture.

\begin{conjecture}\label{conj:prime square}
    If $n>21679$ is an integer that is not a square, then $n=p+x^2$, where $p$ is an odd prime and $x$ is a positive integer.
\end{conjecture}

We note that $21679$ is conjectured to be the largest non-square integer that cannot be written as the sum of an odd prime and a positive square (see \eqref{eq:defn of exceptional set E2} below).

Hardy and Littlewood conjectured that there are only finitely many non-square integers that cannot be written as a sum of a prime and a square (see \cite[Conjecture H]{HL1923}). We can use the results from verifying Conjecture \ref{conj:prime square} up to $10^{14}$ to verify Conjecture \ref{conj:prime squares} up to $10^{26}$. We state this progress on Conjecture \ref{conj:prime squares} as a theorem, which we prove at the end of Section \ref{sec:prime two squares}.

\begin{theorem}\label{thm:prime squares}
   If $n$ is an integer satisfying $14<n<10^{26}$, then $n=p+x^2+y^2$, where $p$ is an odd prime and $x, y$ are positive integers.
\end{theorem}

As mentioned above, this process relies on a ``bootstrapping'' procedure, described in the following theorem, which we prove in Section \ref{sec:prime two squares}. 

\begin{theorem}\label{thm:prime two squares}
    Let $N_2$ be a positive integer. Let $\mathcal{E}_2$ be the finite ``exceptional set'' defined in \eqref{eq:defn of exceptional set E2} below. Assume that every positive integer $<N_2$ is a square, is an element of $\ex_2$, or can be represented as the sum of an odd prime and a positive square. 
    
    Let $n$ be a positive integer. If $n\notin \{1, 2, 3, 4, 6, 14\}$ and 
    \begin{equation*}
    n < \left(\frac{N_2}{8}\right)^2,
    \end{equation*}
    then $n=p+x^2+y^2$, where $p$ is an odd prime and $x$ and $y$ are positive integers.
\end{theorem}

One can ask similar questions and make similar conjectures involving cubes instead of squares. For instance, all large integers should be the sum of an odd prime and two positive cubes.

\begin{conjecture}\label{conj:prime cubes}
    If $n>308$ is a positive integer, then $n=p+x^3+y^3$, where $p$ is an odd prime and $x,y$ are positive integers.
\end{conjecture}

It is far out of reach of present technology to prove that all sufficiently large integers are the sum of a prime and two cubes. However, one can use the circle method to show that all sufficiently large integers may be written as the sum of a prime and four cubes (see \cite[p. 2850]{BW2026}).

Conjecture \ref{conj:prime square} has an analogue for cubes.

\begin{conjecture}\label{conj:prime cube}
    If $n>78526384$ is an integer that is not a cube, then $n=p+x^3$, where $p$ is an odd prime and $x$ is a positive integer.
\end{conjecture}

Similarly to Conjecture \ref{conj:prime square} above, we believe $78526384$ is the largest non-cube integer that cannot be written as the sum of an odd prime and a positive cube. (See Section \ref{sec:prime two cubes} below.)

By verifying Conjecture \ref{conj:prime cube} up to $10^{12}$, we can prove Conjecture \ref{conj:prime cubes} up to $10^{16}$. Due to some Diophantine difficulties, we do not quite have a nice theorem for cubes like Theorem \ref{thm:prime two squares} for squares (see Section \ref{sec:cubes} for further discussion). However, with additional computational work, we can use Theorem \ref{thm:prime two cubes} below for this deduction.

Much of our work in this paper is computational and relies on computer calculation. All the code for this paper is available at our “Sums of Primes with Squares or Cubes” GitHub repository\cite{code}.

The outline of the rest of the paper is as follows. In Section \ref{sec:prime two squares}, we discuss sums of a prime and two squares. In Section \ref{sec:prime two cubes}, we discuss sums of a prime and two cubes. We summarize our computational methods in Section \ref{sec:computational methods}. Then, in Section \ref{sec:cubes}, we summarize our attempts to address the shortcomings in Theorem \ref{thm:prime two cubes}. We discuss some conjectures that, if true, would allow us to strengthen the conclusion of Theorem \ref{thm:prime two cubes}.

\section{Sums of a Prime and Two Squares}\label{sec:prime two squares}
In this section we discuss our results involving sums of an odd prime and two positive squares. However, we begin with sums of an odd prime and one positive square. 

We define the exceptional set $\ex_2$ by
\begin{align}\label{eq:defn of exceptional set E2}
    \ex_2 = \{&2, 3, 5, 10, 13, 31, 34, 37, 58, 61, 85, 91, 127, 130, 214, 226, 370, 379, \nonumber \\ 
    &439, 526, 571, 706, 730, 771, 829, 991, 1255, 1351, 1414, 1549, 1906, \\
    &2986, 3319, 3676, 7549, 9634, 21679\}.\nonumber
\end{align}
We conjecture that $\mathcal{E}_2$ is the complete list of positive non-square integers that cannot be written as the sum of an odd prime and a positive square. The set $\mathcal{E}_2$ appears in Theorem \ref{thm:prime two squares} above, and also in the following theorem.

\begin{theorem}\label{thm:prime square}
Let $n$ be a positive integer with $n<10^{14}$. Assume that $n\notin\ex_2$ and that $n$ is not a square. Then $n=p+x^2$, where $p$ is an odd prime and $x$ is a positive integer.
\end{theorem}

We proved Theorem \ref{thm:prime square} directly with efficient computation. (See our code\footnote{\url{https://github.com/kapplegate2020/Sums-of-a-Prime-with-Squares-or-Cubes/blob/main/prime_power_code/square.cpp}} and also Section \ref{subsec:direct computations} for more details.)

Now we move on to Theorem \ref{thm:prime two squares}, which, combined with the previous theorem will allow us to arrive at our main computational result in Theorem \ref{thm:prime squares}. We first prove two lemmas.

\begin{lemma}\label{lem:four squares}
    If $n$ is an integer, then $n$ cannot be written as
    \[
    n = t^2+a^2 = (t-1)^2+b^2 = (t-2)^2+c^2 = (t-3)^2+d^2,
    \]
    where $t, a, b, c, d$ are all integers.
\end{lemma}

\begin{proof}
    Let $n$ be an integer, and assume by way of contradiction that 
    \[
    n = t^2+a^2 = (t-1)^2 + b^2 = (t-2)^2+c^2 = (t-3)^2+d^2
    \]
    for some integers $t, a, b, c, d$.

    We work modulo 8. A straightforward computation reveals that 0, 1, and 4 are the only quadratic residues modulo 8. Of the values $t, t-1, t-2,$ and $t-3$, at least one must be divisible by 4, at least one must be odd, and at least one must be $\equiv 2 \pmod{4}$. Assume that $t$ is divisible by 4, so that $t-1$ is odd and $t-2\equiv 2 \pmod{4}$ (the arguments in the other cases are similar). We can then compute
    \begin{align*}
        n \equiv t^2+a^2 \equiv 0+a^2 &\equiv 0, 1, \text{or }4 \pmod 8\\
        n \equiv (t-1)^2+b^2 \equiv 1+b^2 &\equiv 1, 2, \text{or }5 \pmod 8\\
        n \equiv (t-2)^2+c^2 \equiv 4+c^2 &\equiv 4, 5, \text{or }0 \pmod 8.
    \end{align*}
    This is a contradiction, since no residue appears in all three of the above equations as possibilities for $n\pmod 8$. 
\end{proof}

\begin{lemma}\label{lem:exception three squares}
    Let $n\geq 100$ be a positive integer, and set $t = \lfloor \sqrt{n}\rfloor$. Assume
    \begin{equation}\label{eq:exception three squares}
        n = t^2+e = (t-1)^2 + a^2 = (t-2)^2+b^2 = (t-3)^2+c^2
    \end{equation}
    for some $e\in \ex_2$ and some integers $a,b,c$. Then $n=66625$.
\end{lemma}

\begin{proof}
    Let $n$ satisfy \eqref{eq:exception three squares}. By subtracting $t^2$, we have
    \[
    e = -2t+1+a^2 = -4t+4+b^2 = -6t+9+c^2.
    \]
    Rearranging and solving, we get
    \[
        a^2b^2c^2 = (2t+e-1)(4t+e-4)(6t+e-9),
    \]
    and expanding yields
    \[  
        (abc)^2 = 48t^3+(44e-144)t^2+(12e^2-96e+132)t+(e^3-14e^2+49e-36).
    \]
    We multiply both sides by 36 to obtain a monic expression on the right-hand side:
    \[
        (6abc)^2 = (12t)^3+(11e-36)(12t)^2+3(12e^2-96e+132)(12t)+36(e^3-14e^2+49e-36).
    \]
    Thus, the point $(12t, 6abc)$ is an integral point on the elliptic curve 
    \[
    E_e:y^2=x^3+(11e-36)x^2+3(12e^2-96e+132)x+36(e^3-14e^2+49e-36).    \]
    We go through each $e \in \mathcal{E}_2$ and compute the integral points on $E_e$ in SageMath\footnote{\url{https://github.com/kapplegate2020/Sums-of-a-Prime-with-Squares-or-Cubes/blob/main/prime_square_square_code/lemma2_3.sage}}. We note that not every integral point $(x,y)$ on $E_e$ corresponds to a solution to \eqref{eq:exception three squares}, since we must have $12 \mid x$ and $6 \mid y$. Additionally, from equation \eqref{eq:exception three squares}, we can deduce that $c^2>b^2>a^2>e>0$, so we may choose to only look for $a, b, c>0$, which implies that $y=6abc>0$. Further, there are cases when the product
    \[
    (n-(t-1)^2)(n-(t-2)^2)(n-(t-3)^2)
    \]
    is a square, although two or more of the individual terms may not be squares. After doing the computations, we find that the only integer $n$ that satisfies equation \eqref{eq:exception three squares} is $n=66625$, arising from $e = 61$:
    \begin{align*}
        66625 &= 258^2+61 = 257^2 + 24^2 = 256^2+33^2=255^2+40^2. \qedhere
    \end{align*}
\end{proof}

Now that we have proved the lemmas, we can prove Theorem \ref{thm:prime two squares}.

\begin{proof}[Proof of Theorem \ref{thm:prime two squares}]
    Let $n$ be a positive integer with $n\notin\{1, 2, 3, 4, 6, 14\}$. For all such $n\leq1.2\cdot10^8$, we verify directly that $n$ can be written as the sum of an odd prime and two positive squares.\footnote{\url{https://github.com/kapplegate2020/Sums-of-a-Prime-with-Squares-or-Cubes/blob/main/prime_square_square_code/theorem1_4.py}} Therefore, assume that $1.2\cdot10^8<n<(\frac{N_2}{8})^2$, where $N_2$ is as in the statement of the theorem. Let $t=\left \lfloor\sqrt{n}\right \rfloor$, and note that $t>10900$ since $n > 1.2 \cdot 10^8$. 
    
    Consider the values
    \begin{align*}
        \alpha_0 = n-t^2, \ \ \ \ \alpha_1 = n-(t-1)^2,\ \ \ \ \alpha_2 = n-(t-2)^2, \ \ \ \ \alpha_3 = n-(t-3)^2.
    \end{align*}
    It is clear that $0\leq \alpha_0<\alpha_1 < \alpha_2 < \alpha_3$, and
    \begin{align*}
        \alpha_3 \leq (t+1)^2-(t-3)^2 < 8t \leq 8\sqrt{n} < N_2.
    \end{align*}
    Further, if $1 \leq i \leq 3$, we have
    \begin{align*}
        \alpha_i \geq \alpha_1 = n-t^2+2t-1 \geq 2t-1 > 2\cdot 10900-1 > 21679,
    \end{align*}
    so we cannot have any of $\alpha_1,\alpha_2,\alpha_3$ be elements of $\mathcal{E}_2$. 
    
    If $\alpha_0 = 0 = 0^2$, then by Lemma \ref{lem:four squares} we see that one of $\alpha_1,\alpha_2,\alpha_3$ is not a square. Since these integers are positive and not in $\mathcal{E}_2$, then by assumption one of the $\alpha_i$ with $1 \leq i \leq 3$ is the sum of an odd prime and a positive square. We then have the desired result for $n$.

    Now assume $\alpha_0$ is positive. By assumption, we have that every $\alpha_i$ is a square, is an element of $\mathcal{E}_2$, or is the sum of an odd prime and a square. If $\alpha_0 \in \mathcal{E}_2$ and the rest are squares, then by Lemma \ref{lem:exception three squares} we have $n=66625$. Since $66625 = 23+119^2+229^2$, say, we see that $66625$ is the sum of an odd prime and two positive squares. Otherwise, one of the $\alpha_i$ is the sum of an odd prime and a positive square, and it follows that $n$ is the sum of an odd prime and two positive squares.
\end{proof}

Theorem \ref{thm:prime two squares} is significant, because if Conjecture \ref{conj:prime square} is verified to a greater upper bound, then the theorem automatically yields a greater upper bound for Conjecture \ref{conj:prime squares} with no additional work. This is exemplified in the proof of Theorem \ref{thm:prime squares}.

\begin{proof}[Proof of Theorem \ref{thm:prime squares}]
By Theorem \ref{thm:prime square}, we can apply Theorem \ref{thm:prime two squares} with $N_2 = 10^{14}$. Since
\begin{align*}
    \left(\frac{10^{14}}{8} \right)^2 = 1.5625 \cdot 10^{26} > 10^{26},
\end{align*}
we obtain the result.
\end{proof}

\section{Sums of a Prime and Two Cubes}\label{sec:prime two cubes}
We now shift from working with squares to working with cubes. In a similar manner to the previous section, we first look at sums of an odd prime and a positive cube before examining sums of an odd prime and two positive cubes. We define $\ex_3$ to be all the positive non-cube integers that we found that cannot be written as the sum of an odd prime and a positive cube. (Thus, $\mathcal{E}_3$ is the ``exceptional set'' for cubes in the same way that $\mathcal{E}_2$ is the exceptional set for squares.) We conjecture that $\mathcal{E}_3$ is the complete list of positive, non-cube integers that cannot be written as the sum of an odd prime and a positive cube. We do not reproduce $\mathcal{E}_3$ here since it has 7058 elements (the largest of which is 78526384), but see Table \ref{tab:cube exceptions} for some rough details.\footnote{See the full list here:\par  \url{https://github.com/kapplegate2020/Sums-of-a-Prime-with-Squares-or-Cubes/blob/main/exceptions/prime_cube.txt}}

\begin{table}[h]
    \centering
    \begin{tabular}{|c|c|}
        \hline
        \textbf{Range} & \textbf{Number of elements of $\mathcal{E}_3$} \\
        \hline
        1-10,000 & 970\\
        10,001-100,000 & 2011\\
        100,001-1,000,000 & 2808\\
        1,000,001-10,000,000 & 1181\\
        10,000,001-100,000,000 & 88\\
        \hline
    \end{tabular}
    \caption{Number of prime-cube exceptions in each range.}
    \label{tab:cube exceptions}
\end{table}

Now we can state the following theorem. It was proven directly with efficient computation. See our code\footnote{\url{https://github.com/kapplegate2020/Sums-of-a-Prime-with-Squares-or-Cubes/blob/main/prime_power_code/cube.cpp}} or Section \ref{subsec:direct computations} for more details.
\begin{theorem}\label{thm:prime cube}
Let $n$ be a positive integer with $n<10^{12}$. Assume that $n\notin\ex_3$ and $n$ is not a cube. Then $n=p+x^3$, where $p$ is an odd prime and $x$ is a positive integer.
\end{theorem}

We can use Theorem \ref{thm:prime cube} and bootstrap the results to get a better result for the sum of a prime and two cubes. We state this as a general theorem.

\begin{theorem}\label{thm:prime two cubes}
    Let $N_3$ be a positive integer. Assume that all positive integers less than $N_3$ are either a cube, an element of the exceptional set $\ex_3$, or can be represented as the sum of an odd prime and a positive cube. 
    
    Let $n$ be a positive integer. If $n\notin \{1, 2, 3, 4, 6, 8, 10, 11, 17, 18, 24, 30, 34, 36,37, 44, 60, 74, 86, $ $ 90, 93, 100, 130, 210, 252, 308\}$ and 
    \begin{equation*}
    n < \left(\frac{N_3-2}{6}\right)^{3/2},
    \end{equation*}
    then either $n=p+x^3+y^3$, where $p$ is an odd prime and $x$ and $y$ are positive integers, or $n-\lfloor n^{1/3}\rfloor^3$ and $n-(\lfloor n^{1/3}\rfloor-1)^3$ are both cubes.
\end{theorem}

To prove this result, we first state and prove two lemmas.

\begin{lemma}\label{lem:cube}
    Let $n$ be a positive integer and let $t=\lfloor n^{1/3}\rfloor$. There is an explicit finite list of $n$ that satisfy at least one of the following conditions:
    \begin{enumerate}
        \item $n-t^3\in \ex_3$ and $n-(t-1)^3\in \ex_3$.
        \item $n-t^3=a^3$ for some integer $a\geq 0$, and $n-(t-1)^3\in \ex_3$.
    \end{enumerate}
    Furthermore, for all such $n$ with
    \[n\notin\{2, 3,8, 10, 11, 17, 18, 24, 30, 34, 36, 37, 44, 60, 74, 86, 90, 93, 100, 130, 210, 252, 308\},\]
    the integer $n$ can be written as the sum of an odd prime and two positive cubes.
\end{lemma}

\begin{proof}
    Let $n$ be a positive integer, and let $t=\lfloor n^{1/3}\rfloor$. We examine each condition in turn.
    \begin{enumerate}
        \item Let $e_1 = n-t^3$ and $e_2 = n-(t-1)^3$ with $e_1, e_2\in \ex_3$. Note that $e_2 > e_1$, and 
        \[
        e_2-e_1 = n-(t-1)^3 - (n-t^3) = 3t^2-3t+1.
        \]
        Since the set $\ex_3$ is finite, there are only finitely many possibilities for $e_2-e_1$. There are at most finitely many values of $t$ such that $3t^2-3t+1=e_2-e_1$ for any fixed $e_1, e_2\in \ex_3$. Thus, there are only finitely many $n$ for which this condition holds\footnote{See our code here: \url{https://github.com/kapplegate2020/Sums-of-a-Prime-with-Squares-or-Cubes/blob/main/prime_cube_cube_code/lemma3_3/lemma3_3_1.py}}. In particular, we determined that there are $32734$ integers $n$ with this property\footnote{Complete list: \url{https://github.com/kapplegate2020/Sums-of-a-Prime-with-Squares-or-Cubes/blob/main/prime_cube_cube_code/lemma3_3/lemma3_3_1_possible_ns.txt}}. The only such $n$ that cannot be written as the sum of an odd prime and two positive cubes are the following:
\begin{align*}
    n \in \{3, 10, 11, 17, 18, 24, 30, 34, 36, 37, 44, 60, 74, 86, 90, 93, 100, 130, 210, 252, 308\}.
\end{align*}
        
        \item Assume $a^3 = n-t^3$ for some nonnegative integer $a$. Let $e = n-(t-1)^3$ and assume $e\in \ex_3$. Since $n\geq t^3$, we can write
        \[
        e = n-(t-1)^3 \geq t^3-(t-1)^3 = 3t^2-3t+1.
        \]
        However, since $\ex_3$ is finite, $e$ is bounded above. Thus, once $n$ is sufficiently large such that $3t^2-3t+1$ is larger than any element of $\ex_3$, then it is impossible for the above condition to hold for that $n$. Thus, there are only finitely many $n$ for which this condition holds\footnote{See our code here: \url{https://github.com/kapplegate2020/Sums-of-a-Prime-with-Squares-or-Cubes/blob/main/prime_cube_cube_code/lemma3_3/lemma3_3_2.py}}. We computed that there are 766 integers $n$ with this property\footnote{Complete list: \url{https://github.com/kapplegate2020/Sums-of-a-Prime-with-Squares-or-Cubes/blob/main/prime_cube_cube_code/lemma3_3/lemma3_3_2_possible_ns.txt}}. Of these, only $n=2$ and $n=8$ cannot be written as the sum of an odd prime and two positive cubes.\qedhere
    \end{enumerate}
\end{proof}

\begin{lemma}\label{lem: cube, elliptic curve}
    Let $n$ be a positive integer and let $t=\lfloor n^{1/3}\rfloor$. There is an explicit finite list of $n$ such that $n-t^3\in \ex_3$ and $n-(t-1)^3=a^3$ for some positive integer $a$. Furthermore, all such $n$ can be written as the sum of an odd prime and two positive cubes.
\end{lemma}

\begin{proof}
    Let $e = n-t^3$ and assume $e\in \ex_3$. Assume $n-(t-1)^3=a^3$ for some $a\in \Z$. By solving for $n$, we write
    \begin{equation}
    t^3 + e = (t-1)^3 + a^3. \label{eq:exception cube elliptic curve}
    \end{equation}
    By expanding and rearranging, we get
    \begin{equation}\label{eq:elliptic curve}
    3t^2 -3t +(e+1) = a^3.
    \end{equation}
    Multiplying through by 27 to make things monic yields
    \[
    (9t)^2 -9(9t) +27(e+1) = (3a)^3.
    \]
    Thus, a solution to \eqref{eq:exception cube elliptic curve} corresponds to a point $(3a, 9t)$ on the elliptic curve
    \[
    E_e: y^2 -9y = x^3-27(e+1).
    \]
    Since there are only finitely many integral points on each elliptic curve (see, e.g., \cite[Chapter IX]{Sil2009}), and only finitely many $e\in \ex_3$ that we can choose, then there are only finitely many $n$ for which this condition holds\footnote{\url{https://github.com/kapplegate2020/Sums-of-a-Prime-with-Squares-or-Cubes/blob/main/prime_cube_cube_code/lemma3_4/lemma3_4part1.sage}}. We computed that there are at most 1111 such $n$\footnote{Complete list: \url{https://github.com/kapplegate2020/Sums-of-a-Prime-with-Squares-or-Cubes/blob/main/prime_cube_cube_code/lemma3_4/possible_ns.txt}}. All of them, however, can be written as the sum of an odd prime and two positive cubes\footnote{\url{https://github.com/kapplegate2020/Sums-of-a-Prime-with-Squares-or-Cubes/blob/main/prime_cube_cube_code/lemma3_4/lemma3_4part2.py}}. Computing these $1111$ integers $n$ was the most computationally-intensive part of this lemma. See Section \ref{sec:elliptic curves} for more details.
\end{proof}

With the previous lemma in hand, we prove Theorem \ref{thm:prime two cubes}.
\begin{proof}[Proof of Theorem \ref{thm:prime two cubes}]
    Let $N_3$ be as in the statement of the theorem. Let $n$ be a positive integer with 
    \[n\notin \{1, 2, 3, 4, 6, 8, 10, 11, 17, 18, 24, 30, 34, 36, 37, 44, 60, 74, 86, 90, 93, 100, 130, 210, 252, 308 \}\]
    and $n < \left(\frac{N_3-2}{6}\right)^{3/2}$. The result follows by easy calculation if $n \leq 7$, so we may assume $n\geq 8$. We set $t=\lfloor n^{1/3}\rfloor$, and note that $t-1 > 0$.

We first observe that
    \begin{align*}
        n-t^3 <n-(t-1)^3 &< (t+1)^3 - (t-1)^3= 6t^2+2\leq 6n^{2/3}+2< N_3.
    \end{align*}
    Thus, by assumption, each of the values $n-t^3$ and $n-(t-1)^3$ is either a cube, is an element of $\ex_3$, or can be written as the sum of an odd prime and a positive cube. If either can be written as the sum of an odd prime and a positive cube, or if both are cubes, then we are done. The other cases are handled by Lemmas \ref{lem:cube} and \ref{lem: cube, elliptic curve}, since we excluded all the $n$ that cannot be written as the sum of an odd prime and two positive cubes.
\end{proof}

The obvious place where the conclusion of Theorem \ref{thm:prime two cubes} is lacking is the case when $n-t^3$ and $n-(t-1)^3$ are both cubes. (See Section \ref{sec:cubes} for further discussion on this point.) However, we can obtain the following result with some computational effort.
\begin{corollary}
    For all positive integers $308<n<10^{16}$, we have $n=p+x^3+y^3$, where $p$ is an odd prime and $x, y$ are positive integers.
\end{corollary}
\begin{proof}
    This almost follows from Theorem \ref{thm:prime cube} and Theorem \ref{thm:prime two cubes}. It only remains to show that for all $n<10^{16}$, if $t=\lfloor n^{1/3}\rfloor$ and if the integers $n-t^3=a^3$ and $n-(t-1)^3=b^3$ are both cubes, then $n$ can be written as the sum of an odd prime and two positive cubes. We verify this computationally, as follows.

    First, we can loop through all $t$ with $1\leq t\leq \sqrt[3]{10^{16}}$. Then, we can loop through all $a$ within $0\leq a^3< 3t^2+3t+1$, since $t^3+a^3=n < (t+1)^3$. Searching over all such pairs $(t,a)$ corresponding to $n \leq N$ has a complexity of $\approx \frac{3^{4/3}}{5}N^{5/9}$, which is within reach of a brute-force search. 
    
    We speed up the brute-force search using some congruence arguments. We have $t^3+a^3-(t-1)^3=b^3$, so $t^3+a^3-(t-1)^3$ must be a cube modulo $q$ for every positive integer $q$. We found all pairs $(u,v)$ modulo 351 such that $u^3+v^3-(u-1)^3$ is a cube modulo 351, and then for each pair we can impose the conditions $t\equiv u \pmod{351}, a \equiv v \pmod{351}$ in our search. This provides about a fivefold speed up. There are other good choices besides 351 that can be used, but we chose 351 to balance the size of the modulus and the percentage of pairs that can be skipped\footnote{See our code here: \url{https://github.com/kapplegate2020/Sums-of-a-Prime-with-Squares-or-Cubes/blob/main/prime_cube_cube_code/corollary3_5.py}}. 
\end{proof}

\section{Computational Methods}\label{sec:computational methods}

In this section we explore the computational methods that were used throughout the paper. There were two main computational efforts that were needed. The first was the verification of which numbers can be written as the sum of a prime and a square (or cube). The second was working with elliptic curves to deal with the cases in Lemma \ref{lem:exception three squares} and Lemma \ref{lem: cube, elliptic curve}.

\subsection{Remarks on the proofs of Theorems \ref{thm:prime square} and \ref{thm:prime cube}}\label{subsec:direct computations}

In order to speed up the computations in proving Theorem \ref{thm:prime square}, we simultaneously check several integers at once for representability as the sum of an odd prime and a positive square\footnote{See our code here: \url{https://github.com/kapplegate2020/Sums-of-a-Prime-with-Squares-or-Cubes/tree/main/prime_power_code}}. In particular, we make use of bitmaps of 64 bits to verify 64 numbers at once. First, for $k\geq0$ we make a 64-bit bitmap that represents all the odd numbers from $128k$ and $128(k+1)$ and marks the primes with a 1. Then, for each $k\geq0$, we look at the range of odd integers between $128k$ and $128(k+1)$. We compute all the even squares less than $128(k+1)$ and we OR a shifted version of the appropriate prime bitmap with the odd bitmap to mark sums of the even square and all the primes. We continue moving down through even squares until we run out of squares or the entire bitmap is ones. We do a very similar process for all the even integers between $128k$ and $128(k+1)$.

In our computations, we used a list of all the primes up to $10^{10}$. However, as the integers get larger, we eventually reach numbers that cannot be represented as the sum of a positive square  and an odd prime less than $10^{10}$. These values are initially relatively rare, so we deal with them individually, subtracting squares and testing the difference for primality. We used a deterministic version of the Miller--Rabin primality test (essentially due to Sorenson and Webster \cite{SW2017}) since the values we were testing were sufficiently small.

When dealing with cubes in Theorem \ref{thm:prime cube}, we did an almost identical computation. The biggest roadblock with cubes is that we hit many more integers which required a prime larger than $10^{10}$ much sooner. Because of that, we were only able to verify that non-cubes can be written as the sum of a prime and a cube up to $10^{12}$, while we were able to go much further for squares.

We remark that we also looked at sums of an odd prime and a fourth power, and found that there are 215,414,496 positive integers below $10^{10}$ that are exceptions (i.e. they are not fourth powers and cannot be expressed as the sum of a prime and positive fourth powers). Although we have not done an exhaustive search, we have also found exceptions as large as $10^{14}$. We would conjecture that this set is finite, but we are not sure how large it might be.

\subsection{Remarks on the proof of Lemma \ref{lem: cube, elliptic curve}}\label{sec:elliptic curves}
In most cases, we used the SageMath elliptic curve method \texttt{integral\_points()} to compute the integer points of various elliptic curves to verify Lemma \ref{lem:exception three squares} and Lemma \ref{lem: cube, elliptic curve}. However, in 865 of the 7058 elliptic curves involved in the proof of Lemma \ref{lem: cube, elliptic curve}, the method \texttt{integral\_points()} failed to successfully compute a Mordell-Weil basis, and thus fails to accurately find all integer points. This failure happened more frequently with elliptic curves with larger coefficients. In this case, we reverted to a method described by Bennett and Ghadermarzi \cite{Bennett_Ghadermarzi_2015} to find integral points on Mordell curves\footnote{See our code here: \url{https://github.com/kapplegate2020/Sums-of-a-Prime-with-Squares-or-Cubes/blob/main/prime_cube_cube_code/lemma3_4/mordellCurve.sage}}. We refer the reader to that paper for more precise details than we give below.

Recall that a Mordell curve is an elliptic curve of the form
\begin{equation}\label{eq:mordell curve}
y^2 = x^3+K,
\end{equation}
where $K$ is a nonzero integer. Recall also from \eqref{eq:elliptic curve} that $3t^2 -3t +(e+1) = a^3$, where $e\in \ex_3$ and $t$ and $a$ are positive integers. If we complete the square and multiply through by $1728$ to clear the leading coefficients, we get 
\[
(36(2t-1))^2 = (12a)^3-432(4e+1).
\]
So, by substituting $y=36(2t-1)$, $x=12a$ and $K=-432(4e+1)$, we have an equation in the form of \eqref{eq:mordell curve}. To solve the equation using the method described by Bennett and Ghadermarzi, we must find all binary cubic forms 
\begin{align}\label{eq:F binary cubic form}
F(x, y)=ax^3+3bx^2y+3cxy^2+dy^3    
\end{align}
with discriminant $D=-108K$, where $D$ is given as
\begin{equation}\label{eq:discriminant}
D = -27(a^2d^2-6abcd-3b^2c^2+4ac^3+4b^3d).
\end{equation}
Furthermore, for each binary cubic form, we must then solve an associated cubic Thue equation (see \cite[Section 2]{Bennett_Ghadermarzi_2015} for the precise correspondence). We note that $F$ has positive discriminant, since $K = -432(4e+1)$ is negative.

The cubic forms $F$ need not be irreducible, so we compute reducible and irreducible forms separately.

Every irreducible binary cubic form as in \eqref{eq:F binary cubic form} is $\text{GL}_2(\mathbb{Z})$-equivalent to a unique reduced form \cite[Proposition 3.2]{Bennett_Ghadermarzi_2015}, and one can determine explicit bounds on the variables $a,b,c,d$ of a reduced form \cite[Lemma 3.3]{Bennett_Ghadermarzi_2015}. Each variable is an element of an interval of length $\lessapprox |K|^{1/4}$. We fix $K$ and loop over possible values of $a,b$, and $c$. With $a,b,c,K$ all fixed, \eqref{eq:discriminant} is a quadratic equation in $d$ and hence $d$ is easily determined. Thus, for fixed $K$, the total complexity of finding all irreducible cubic forms as in \eqref{eq:F binary cubic form} with discriminant $-108K$ is $\lessapprox |K|^{3/4}$.

Because of the structure of $D$, we know that it is divisible by $108 \cdot 432=46656=2^6\cdot 3^6$. However, the $-27$ already contributes a $3^3$ to $D$. Thus, by plugging in all the possibilities of $a, b, c, d$ modulo $2^6 \cdot 3^3$, we can see which tuples could produce a $D$ that is divisible by $46656$. To do this, we construct a dictionary whose keys are tuples $(a, b) \pmod{2^6}$ and whose value is a list of possible residues of $c\pmod{2^6}$. To compute this dictionary, we went through every tuple $(a, b, c)\pmod{2^6}$. We then loop through every $d\pmod{2^6}$ and check if there is some $d$ such that \eqref{eq:discriminant} can be satisfied modulo $2^6$. If there is, then we add $c$ to the list in the dictionary keyed by $(a, b)$. Once we have the list of possible $c$ residues, we can sometimes reduce the list to save space. For example, we could collapse the residue list $\{0, 1, 16, 17, 32, 33, 48, 49\}\pmod{64}$ to the list $\{0, 1\} \pmod{16}$ to save space and time.

We can perform a similar operation modulo $3^3$ to generate another dictionary. Then, as we loop through every $a$ and $b$ in the brute force, we use the congruence conditions to skip some values of $c$. In practice, only about $28.6\%$ of the congruence classes for $c$ survived, which produces about a three times speed up. 

A reducible binary cubic form $F$ with discriminant $D = -108K$ as in \eqref{eq:F binary cubic form} is $\text{SL}_2(\mathbb{Z})$-equivalent to $f(x,y) = x(x^2+3Bxy+3Cy^2)$ for some integers $B$ and $C$ with $D = 27C^2(3B^2-4C)$ (see \cite[(3.1) and (3.2)]{Bennett_Ghadermarzi_2015}). The variable $C$ lives in an interval of length $\lessapprox |K|^{1/2}$ (see \cite[Lemma 3.7]{Bennett_Ghadermarzi_2015}). When $K$ and $C$ are fixed, we may then find $B$ by solving a quadratic equation, so the complexity of determining all reducible forms is $\lessapprox |K|^{1/2}$.

Once we find all the reducible and irreducible forms, then we can use the built-in Sage method \texttt{gp.thueinit} to set up and solve the related Thue equation. Following the work in Bennett and Ghadermarzi paper, we can solve for the integral points.

\section{Cube problems}\label{sec:cubes}
Theorem \ref{thm:prime two squares} is a relatively clean ``bootstrapping'' result for squares. A key point in the proof is that, given a large $n$ and $t = \lfloor \sqrt{n}\rfloor$, not all of $n-t^2,n-(t-1)^2,n-(t-2)^2,n-(t-3)^2$ can be squares. Similarly, the conclusion of Theorem \ref{thm:prime two cubes} would be greatly enhanced if one could prove the following conjecture.

\begin{conjecture}
    Let $n>8$ be a positive integer, and let $t=\lfloor n^{1/3}\rfloor$. Then $n-t^3$, $n-(t-1)^3$, and $n-(t-2)^3$ are not all cubes.
\end{conjecture}

Perhaps one could even prove the following stronger conjecture.
\begin{conjecture}\label{conj:3 cubes}
    For any positive integers $a, b, c, t$, the integers $t^3+a^3$, $(t-1)^3+b^3$, and $(t-2)^3+c^3$ are not all equal.
\end{conjecture}

If we drop the condition in Conjecture \ref{conj:3 cubes} that the integers are positive, then there are trivial solutions with $a=-t$, $b=-(t-1)$ and $c=-(t-2)$. However, there is also a ``nontrivial'' solution with $t=6, a=-3, b=4, c=5$. We were unable to find any other nontrivial solutions. A simple brute force over $t$ and $a$ shows that there are no counterexamples to Conjecture \ref{conj:3 cubes} with $t<10^{4}$.

One natural approach in attempting to prove Conjecture \ref{conj:3 cubes} is to try to parameterize solutions to
\begin{equation}\label{eq:cubes}
t^3+a^3=(t-1)^3+b^3 
\end{equation}
and then somehow show that $t^3+a^3-(t-2)^3$ is not a cube. In our efforts to parameterize these curves, we did find some polynomially-parameterized families, but these families did not account for all the solutions to \eqref{eq:cubes}.

The first, and simplest, family that we found was the one-parameter family
\begin{align}\label{eq:one param family}
    (3x^3-3x^2+2x)^3 + x^3 = (3x^3-3x^2+2x-1)^3+(3x^2-2x+1)^3
\end{align}
where $x$ is a positive integer. In light of Conjecture \ref{conj:3 cubes}, it is natural to ask whether
\[
(3x^3-3x^2+2x)^3 + x^3 - (3x^3-3x^2+2x-2)^3
\]
can ever be a cube. This is essentially equivalent to determining all the integral points on the superelliptic curve
\begin{align*}
    c^3 &= (3x^3-3x^2+2x)^3 + x^3 - (3x^3-3x^2+2x-2)^3\\
    &= 54x^6-108x^5+126x^4-107x^3+60x^2-24x+8,
\end{align*}
which is already a nice challenge. It seems likely the only integral point is $(x,c) = (0,2)$.

With a great deal of searching, we were able to expand this one-parameter family into a two-parameter family, where
\begin{align*}
t &= t_3x^3+t_2x^2+t_1x+t_0,\ \ \ \ \ a = a_2x^2+a_1x+y, \ \ \ \ \ b = 3(t_0^3+1)x^2+3t_0^2x+t_0,
\end{align*}
and
\begin{align*}
    t_3 &= 9y^4+18y^3+27y^2+18y+9, & a_2 &= 3y^3-3,\\
    t_2 &= 12y^3+18y^2+18y+6, & a_1 &= 3y^2,\\
    t_1 &= 6y^2+6y+3, & t_0 &= y+1,
\end{align*}
where $x$ and $y$ are positive integers. To get the original family \eqref{eq:one param family} from this extended family, we first set $y=1$. This gives
\begin{align*}
    t = 81x^3 + 54x^2 + 15x + 2, \ \ \ \ a = 3x+1, \ \ \ \ b = 27x^2 + 12x + 2.
\end{align*}
We then set $u=3x+1$ to get $t = 3u^3-3u^2+2u, a=u, b = 3u^2-2u+1$.

\section*{Acknowledgments}

The second author is partially supported by the National Science Foundation (DMS-2418328) and the Simons Foundation (MPS-TSM-00007959).

We thank the Office of Research Computing at Brigham Young University for providing us with access to high-performance computing resources. We used those resources to prove Theorem \ref{thm:prime square}, Theorem \ref{thm:prime cube}, and Lemma \ref{lem: cube, elliptic curve}.

We used ChatGPT 5.6 Sol for proofreading, checking computer code for bugs, and catching errors in earlier versions of the manuscript.

\bibliographystyle{plain}
\bibliography{refs.bib}

\begin{thebibliography}{10}

\bibitem{code}
K.~Applegate.
\newblock Sums of a prime with squares or cubes.
\newblock
  \url{https://github.com/kapplegate2020/Sums-of-a-Prime-with-Squares-or-Cubes},
  2026.
\newblock GitHub repository, accessed June 2026.

\bibitem{Bennett_Ghadermarzi_2015}
M.~A. Bennett and A.~Ghadermarzi.
\newblock Mordell’s equation: a classical approach.
\newblock {\em LMS Journal of Computation and Mathematics}, 18(1):633–646,
  2015.

\bibitem{BW2026}
J.~Br\"udern and T.~D. Wooley.
\newblock Partitio {N}umerorum: sums of a prime and a number of {$k$}-th
  powers.
\newblock {\em J. Eur. Math. Soc. (JEMS)}, 28(7):2849--2875, 2026.

\bibitem{Dav2000}
H.~Davenport.
\newblock {\em Multiplicative number theory}, volume~74 of {\em Graduate Texts
  in Mathematics}.
\newblock Springer-Verlag, New York, third edition, 2000.
\newblock Revised and with a preface by Hugh L. Montgomery.

\bibitem{HL1923}
G.~H. Hardy and J.~E. Littlewood.
\newblock Some problems of `{P}artitio numerorum'; {III}: {O}n the expression
  of a number as a sum of primes.
\newblock {\em Acta Math.}, 44(1):1--70, 1923.

\bibitem{Helfgott}
H.~A. Helfgott.
\newblock The ternary {G}oldbach problem, 2015.
\newblock https://arxiv.org/abs/1501.05438.

\bibitem{Hooley1957}
C.~Hooley.
\newblock On the representation of a number as the sum of two squares and a
  prime.
\newblock {\em Acta Math.}, 97:189--210, 1957.

\bibitem{Linnik1963}
Ju.~V. Linnik.
\newblock {\em The dispersion method in binary additive problems}.
\newblock American Mathematical Society, Providence, RI, 1963.
\newblock Translated by S. Schuur.

\bibitem{Goldbach2014}
T.~Oliveira~e Silva, S.~Herzog, and S.~Pardi.
\newblock Empirical verification of the even {G}oldbach conjecture and
  computation of prime gaps up to {$4\cdot 10^{18}$}.
\newblock {\em Math. Comp.}, 83(288):2033--2060, 2014.

\bibitem{Siksek2016}
S.~Siksek.
\newblock Every integer greater than 454 is the sum of at most seven positive
  cubes.
\newblock {\em Algebra Number Theory}, 10(10):2093--2119, 2016.

\bibitem{Sil2009}
J.~H. Silverman.
\newblock {\em The arithmetic of elliptic curves}, volume 106 of {\em Graduate
  Texts in Mathematics}.
\newblock Springer, Dordrecht, second edition, 2009.

\bibitem{SW2017}
J.~Sorenson and J.~Webster.
\newblock Strong pseudoprimes to twelve prime bases.
\newblock {\em Math. Comp.}, 86(304):985--1003, 2017.

\bibitem{VW2000}
R.~C. Vaughan and T.~D. Wooley.
\newblock Waring's problem: a survey.
\newblock In {\em Number theory for the millennium, {III} ({U}rbana, {IL},
  2000)}, pages 301--340. A K Peters, Natick, MA, 2002.

\bibitem{Vin1937}
I.~M. Vinogradov.
\newblock Some theorems concerning the theory of primes.
\newblock {\em Rec. Math. Moscou, n. Ser.}, 2:179--195, 1937.

\bibitem{Vin2004}
I.~M. Vinogradov.
\newblock {\em The method of trigonometrical sums in the theory of numbers}.
\newblock Dover Publications, Inc., Mineola, NY, 2004.
\newblock Translated from the Russian, revised and annotated by K. F. Roth and
  Anne Davenport, Reprint of the 1954 translation.

\end{thebibliography}

\end{document}